\documentclass[11pt]{amsart}

\usepackage{amsmath,amssymb, amsthm,graphicx,fancyhdr,lastpage,multicol}
\usepackage{setspace}
\usepackage[latin1]{inputenc}
\usepackage{amsmath}
\usepackage{amsfonts}
\usepackage{amssymb}
\usepackage{fullpage}
\usepackage{color}
\usepackage{enumitem}
\usepackage[all]{xy}

\numberwithin{equation}{section}

\theoremstyle{plain}
\newtheorem{thm}[equation]{Theorem}

\newtheorem{lem}[equation]{Lemma}
\newtheorem{cor}[equation]{Corollary}
\newtheorem{prop}[equation]{Proposition}

\theoremstyle{definition}
\newtheorem{defn}[equation]{Definition}
\newtheorem{que}[equation]{Question}

\theoremstyle{remark}
\newtheorem{rmk}[equation]{Remark}
\newtheorem{eg}[equation]{Example}

\numberwithin{equation}{section}

\DeclareMathOperator{\Div}{Div}
\DeclareMathOperator{\nm}{N}
\DeclareMathOperator{\NS}{NS}
\DeclareMathOperator{\ord}{ord}
\DeclareMathOperator{\Pic}{Pic}
\DeclareMathOperator{\Rank}{Rank}
\DeclareMathOperator{\Spec}{Spec}
\DeclareMathOperator{\tr}{tr}

\renewcommand{\AA}{\mathbb{A}}

\newcommand{\QQ}{\mathbb{Q}}
\newcommand{\ZZ}{\mathbb{Z}}

\newcommand{\I}{\mathcal{I}}

\newcommand{\PP}{\mathbb{P}}
\newcommand{\FF}{\mathbb{F}}
\newcommand{\CC}{\mathbb{C}}

\newcommand{\E}{\mathcal{E}}

\author{Christopher Davis}
\address{University of Copenhagen, Dept of Mathematical Sciences,  
2100 K{\o}benhavn {\O}, Denmark}
\email{davis@math.ku.dk}
\date{\today}

\author{Tommy Occhipinti}
\address{Carleton College, Dept of Mathematics, Northfield, MN 55057}
\email{tocchipinti@carleton.edu }

\title{Explicit points on $y^2 + xy - t^d y = x^3$ and related character sums}

\begin{document}

\maketitle

\begin{abstract}
Let $\FF_q$ denote a finite field of characteristic~$p \geq 5$ and let $d = q+1$.  Let $E_d$ denote the elliptic curve over the function field $\FF_{q^2}(t)$ defined by the equation $y^2 + xy - t^d y = x^3$.  Its rank is $q$ when $q \equiv 1 \bmod 3$ and its rank is $q-2$ when $q \equiv 2 \bmod 3$.  We describe an explicit method for producing points on this elliptic curve.  In case $q \not\equiv 11 \bmod 12$, our method produces points which generate a full-rank subgroup.  Our strategy for producing rational points on $E_d$ makes use of a dominant map from the degree~$d$ Fermat surface over $\FF_{q^2}$ to the elliptic surface associated to $E_d$.  We in turn study lines on the Fermat surface $\mathcal{F}_d$ using certain multiplicative character sums which are interesting in their own right.  In particular, in the $q \equiv 7 \bmod 12$ case, a character sum argument shows that we can generate a full-rank subgroup using $\mu_d$-translates of a single rational point.
\end{abstract}

\section{Introduction}

Let $\FF_q$ denote a finite field of characteristic~$p \geq 5$, let $d = q+1$, and let $\chi$ denote a non-trivial multiplicative character on $\FF_{q^2}$ of order dividing $d$.  Put $\chi(0) = 0$.  In this paper we analyze some character sums of the form 
\begin{equation} \label{general character sum}
\sum_{x \in \FF_{q^2}} \chi(f(x))
\end{equation}
where $f(x) \in \FF_q[x]$ is a separable cubic polynomial.  From our analysis, we derive consequences concerning rational points on certain elliptic curves over the function field $\FF_{q^2}(t)$.  We begin by considering the more elementary case that the polynomial $f(x)$ in (\ref{general character sum}) is a quadratic polynomial.  

Let $f(x) \in \FF_q[x]$ denote any separable quadratic polynomial.  Note that because we force the coefficients to lie in $\FF_q$, the polynomial splits in $\FF_{q^2}[x]$, and (\ref{general character sum}) reduces to a Jacobi sum.  Our assumption that the order of $\chi$ divides $d= q+1$ further simplifies the situation, and well-known results (see for example \cite[Theorems~5.16 and 5.21]{LN97}) then imply
\[
\sum_{x \in \FF_{q^2}} \chi(f(x)) = \begin{cases} q &\mbox{if $\chi$ has order $> 2$} \\-1 &\mbox{if $\chi$ has order 2.} \end{cases}
\]

The condition that $\chi$ has order dividing $d$ implies in particular that the character sum (\ref{general character sum}) is always real.  When $f(x)$ is quadratic, Weil's Theorem for multiplicative character sums \cite[Theorem~5.41]{LN97} implies that (\ref{general character sum}) is bounded in absolute value by $q$.  We can interpret the above result as saying that when $\chi$ has order strictly greater than $2$ and $f(x)$ is a separable quadratic polynomial with coefficients in $\FF_q$, then (\ref{general character sum}) always attains its upper bound.  

In this paper, considerations from arithmetic geometry lead us to consider the following analogue.  Here the polynomial $f(x)$ is cubic and the Weil bound in this case becomes $2q$.  

\begin{que}
Let $f(x) \in \FF_q[x]$ denote a separable cubic polynomial and assume furthermore that $f(x)$ splits completely in $\FF_q[x]$.  Let $\chi$ denote a multiplicative character of order dividing $d := q+1$.  Across all such polynomials $f(x)$, how often does the character sum (\ref{general character sum}) attain its upper bound $2q$?
\end{que}

It is not difficult to attain the following bound; see Section~\ref{props of char sum} for details.

\begin{thm}  \label{bound on N} Assume $\chi$ has order greater than $2$ and dividing $d$.  Let $N$ denote the number of values of $c \in \FF_q$ such that the sum (\ref{general character sum}) attains the upper bound $2q$ for $f(x) = x(x+1)(x+c)$.   Then
\[
N \leq \frac{3q-9}{4}.
\]
\end{thm}

\begin{proof}
Set 
\[
S_c := \sum_{x \in \FF_{q^2}} \chi(x(x+1)(x+c)).
\]
One checks that $\sum_{c \in \FF_q} S_c = q(q-3)$ and $S_0 = S_1 = q$.  By Weil's theorem, $S_c \geq -2q$ for all $c$.  The theorem is proved by considering
\[
q(q-3) = \sum_{c \in \FF_q} S_c \geq q + q + N \cdot 2q + (q-2-N) \cdot -2q = q(6 + 4N - 2q).  
\] 
\end{proof}

\begin{rmk}
In general, the bound of this theorem cannot be strengthened.  We have used Sage to verify we have $N = (3p-9)/4$ for all primes $p \equiv 3 \bmod 4$ between $7$ and $139$ in the specific case that $\chi$ has order 4.  
For $c \in \FF_{p}$, $c \neq 0,1$, computations suggest that $S_c = -2p$ when both $c-1$ is a quadratic residue and $c$ is a quadratic nonresidue in $\FF_{p}$, and $S_c = 2p$ in all other cases.  The claim $N = (3q-9)/4$ continues to hold for $q = 7^3$ and $\chi$ of order 4, however our description in terms of quadratic residues and nonresidues no longer holds in this prime power case.  
\end{rmk}

\begin{rmk}
There is no loss in generality in assuming our polynomial $f(x)$, which splits completely in $\FF_q[x]$, has the special form $x(x+1)(x+c)$.  If we begin with a polynomial of the form $(x+c_1)(x+c_2)(x+c_3)$, we may first replace $x+c_1$ by $x$, and then because $\chi(a) = 1$ for all $a \in \FF_{q}^*$, we can rescale again to reach the desired form.  

One may obtain similar results without the requirement that $f(x)$ split completely in $\FF_q[x]$.  They are proved in a completely analogous way: sum across all such polynomials and apply the Weil bound.  We omit these results as they are not used in the applications below.
\end{rmk}

We now explain a specific context from arithmetic geometry in which these character sums arise.  
Let $E_d$ denote the elliptic curve over the function field $\FF_{q^2}(t)$ defined by the equation 
\begin{equation} \label{the elliptic curve}
E_d : \quad y^2 + xy - t^d y = x^3,
\end{equation}
and let $\mathcal{F}_d$ denote the degree~$d$ Fermat surface over $\FF_{q^2}$ defined by the equation
\[
\mathcal{F}_d: \quad x_0^d + x_1^d + x_2^d + x_3^d = 0.
\]
Our primary strategy is to relate lines on $\mathcal{F}_d$ to rational points on $E_d$ using an explicit rational map between $\mathcal{F}_d$ and a certain affine surface related to $E_d$.  This strategy is similar to the strategy carried out in \cite{Ulm02}.  However, that paper considers Jacobi sums, which have been more well-studied than the character sums we consider.  The character sums in \cite[\S7.5]{Ulm02} describe Frobenius actions on cohomology, while ours describe intersection pairings.  The difference arises because \cite{Ulm02} does not seek to produce explicit points.  (The elliptic curves considered in \cite{Ulm02} are also different from our curves $E_d$.)  

The description of the following family of lines on $\mathcal{F}_d$ is taken from \cite[\S5.2]{SSV10}.

\begin{defn} \label{standard line def}
Let $a \in \FF_{q}$, $b \in \FF_{q^2} \setminus \FF_q$ be elements such that $a^2 +1 = b^2$.  Let $L_{a,b}$ denote the line on $\mathcal{F}_d$ parametrized as follows:
\[
L_{a,b}: \quad \PP^1_{\FF_{q^2}} \rightarrow \mathcal{F}_d, \quad [u:v] \mapsto [u:v:au+bv:av + bu].
\] 
It will also be convenient to have the following reparametrization of $L_{a,b}$:
\[
L_{a,b}: \quad \PP^1_{\FF_{q^2}} \rightarrow \mathcal{F}_d, \quad [s_0:s_1] \mapsto [s_0:\alpha s_0 + \beta s_1: -\alpha s_1 - \beta s_0: s_1],
\]
where $\alpha := -a^{-1}b$ and $\beta := a^{-1}$.  
\end{defn}

Next we give an explicit ring-theoretic procedure for producing a point in $E_d(\FF_{q^2}(t))$ from such a line $L_{a,b}$. (A different description of this procedure is given in Proposition~\ref{compare two maps}.)  
Over the affine sets $x_3 \neq 0$ and $s_1 \neq 0$, the map of $L_{a,b}$ into $\mathcal{F}_d$ corresponds to the ring map
\begin{equation} \label{parametrized line eqn}
\FF_{q^2}[x_0, x_1, x_2]/(x_0^d + x_1^d + x_2^d + 1) \rightarrow k[s_0], \quad x_0 \mapsto s_0,\, x_1 \mapsto \alpha s_0 + \beta,\, x_2 \mapsto -\alpha - \beta s_0.
\end{equation}
Consider the affine surface over $\FF_{q^2}$ with coordinate ring $\FF_{q^2}[x,y,t]/(y^2 + xy - t^dy - x^3)$; the generic fiber of the natural map from this surface to $\Spec \FF_{q^2}[t]$ is isomorphic to the affine part of $E_d$.  We have a map from the affine piece $x_3 = 1$ on $\mathcal{F}_d$ to this affine surface given in terms of coordinate rings by 
\begin{equation} \label{map to affine surface}
\phi: (x,y,t) \mapsto (-x_0^dx_2^d, -x_0^{2d}x_2^d, x_0x_1x_2).
\end{equation}
Composing (\ref{parametrized line eqn}) with (\ref{map to affine surface}), we find a map $\FF_{q^2}[x,y,t]/(y^2 + xy - t^dy - x^3) \rightarrow \FF_{q^2}[s_0]$ which in particular sends $t$ to $-s_0 (\alpha s_0 + \beta) (\alpha + \beta s_0)$.  View this as a map of $\FF_{q^2}[t]$-modules, and tensor with $\FF_{q^2}(t)$.  This left-hand side is the coordinate ring of the affine piece of our elliptic curve $E_d$ and the right-hand side is $\FF_{q^2}[s_0] \otimes_{\FF_{q^2}[t]} \FF_{q^2}(t)$, where the non-obvious map sends $t \mapsto -s_0(\alpha s_0 + \beta) (\alpha + \beta s_0)$.  Because $-s_0(\alpha s_0 + \beta) (\alpha + \beta s_0) - t$ is an irreducible polynomial in the variable $s_0$ over the coefficient ring $\FF_{q^2}(t)$, this corresponds to a point on $E_d$ with coordinates in $\FF_{q^2}(s_0)$, which is a degree 3 extension of $\FF_{q^2}(t)$.  We are seeking a point in $E_d(\FF_{q^2}(t))$, not one in $E_d(\FF_{q^2}(s_0))$.  To this end, we use the group law on $E_d$ to sum the points corresponding to the three total Galois conjugates (i.e., to ``take the trace'') of the point with coordinates in $\FF_{q^2}(s_0)$.  We denote this correspondence as follows:
\begin{equation} \label{line to point}
L_{a,b} \leadsto \phi_*(L_{a,b}) \in E_d(\FF_{q^2}(t)).
\end{equation}

Consider the automorphism of $\FF_{q^2}(t)$ induced by $t \mapsto \zeta_d t,$ where $\zeta_d$ is a fixed primitive $d$-th root of unity in $\FF_{q^2}$.  This induces an automorphism of $E_d(\FF_{q^2}(t))$.  The main arithmetic theorems of this paper are the following Theorems~\ref{main theorem 1} and \ref{main theorem 2}.  

\begin{thm} \label{main theorem 1}
Assume $q \equiv 7 \bmod 12$.  The group $\mu_d$ of $d$-th roots of unity in $\FF_{q^2}$ acts on $E_d(\FF_{q^2}(t))$ as described above.  Let $\phi_*$ be as in (\ref{line to point}).  Let $b$ denote a primitive $12$-th root of unity in $\FF_{q^2}$ and let $a = b^2$.  The element $\phi_*(L_{a,b})$ and its $\mu_d$-translates together rationally generate $E_d(\FF_{q^2}(t)).$ 
\end{thm}

\begin{thm} \label{main theorem 2}
Assume $q \equiv 1 \bmod 4$.  The points $\phi_*(L_{a,b}) \in E_d(\FF_{q^2}(t))$, across all lines $L_{a,b}$ as in Definition~\ref{standard line def}, together with their $\mu_d$-translates, rationally generate $E_d(\FF_{q^2}(t))$.  In fact, $n-1$ of these lines suffice, where $n$ is the number of positive divisors of $d = q+1$.  
\end{thm}

\begin{eg}
Theorems~\ref{main theorem 1} and \ref{main theorem 2} together account for all cases except for $q \equiv 11 \bmod 12$.  Computations show that the points $\phi_*(L_{a,b})$ and their $\mu_d$-translates are not sufficient to rationally generate $E_d(\FF_{q^2}(t))$ when $q = 11$ or $71$. 
\end{eg}

\begin{eg} \label{example of explicit point}
We apply Theorem~\ref{main theorem 1} in the case $q = 7$.  Choose $a,b \in \FF_{49}$ such that $a = b^2 = 3$.  Using Sage to carry out the procedure described above, we find $\phi_*(L_{a,b}) = P$ where
{\tiny
\begin{align*}
P_x &= \frac{-2t^{14} - 2t^{13} + 3t^{12} + t^{11} + t^9 - t^7 + 2t^5 - t^4 - 3t^3 + 3t^2 + 2t + 2}{-2t^8 + 2t^7 + 3t^6 + 3t^5 + t^4 - 3t^3 - 2t^2 - t - 1} \\
\intertext{and}
P_y &= -\frac{t^{21} + t^{20} - t^{19} + 2t^{18} - t^{16} + 2t^{15} + 2t^{14} - 3t^{13} + 2t^{11} + t^{10} + 2t^9 - 2t^8 - t^7 + t^6 + 2t^4 - 2t^3 + t^2 - 1}{t^{12} + 2t^{11} - t^{10} + 2t^9 + 3t^8 + t^6 - t^4 + 2t^3 - t^2 - 2t + 1}.
\end{align*}}Then Theorem~\ref{main theorem 1} implies that the point $P$, together with its $\mu_d$-translates, generates a full-rank (i.e., rank 7) subgroup of $E(\FF_{49}(t))$.
\end{eg}

Theorems~\ref{main theorem 1} and \ref{main theorem 2} are proved by making careful analysis of a character sum $S_c$ as in (\ref{general character sum}): we compute certain inner products of elements in $\NS(\mathcal{F}_d) \otimes_{\ZZ} \QQ(\zeta_d)$ as being equal to $2q - S_c$, and we deduce rational generation as in Theorems~\ref{main theorem 1} and \ref{main theorem 2} as a consequence of suitably many of these inner products being nonzero.

\subsection*{Notation and conventions} We write $E_d$ for the elliptic curve defined in (\ref{the elliptic curve}) and $\mathcal{E}_d$ (together with a fixed morphism $\mathcal{E}_d \rightarrow \PP^1$) for the associated elliptic surface (as in \cite[\S3.2]{Ulm02} or \cite[Lecture~3,~Proposition~1.1]{Ulm11}).  For $X$ a surface and $U \subseteq X$ an open subset, we write $\NS(U)$ for the image of the composite map $\Div U \rightarrow \Div X \rightarrow \NS(X).$  In Definition~\ref{L1 defn} below, we define a subgroup $L^1(\NS(\mathcal{E}_d)) \leq \NS(\mathcal{E}_d)$.  For any group $G$ and homomorphism $f: G \rightarrow \NS(\mathcal{E}_d)$, we write $L^1 (G)$ for the subgroup of $G$ determined by $f^{-1} (L^1(\NS(\mathcal{E}_d)))$.

\subsection*{Acknowledgments} The authors extend special thanks to Doug Ulmer, for suggesting this project and for giving many valuable suggestions, and to Daqing Wan, for providing explicit character sum arguments which will appear in a subsequent paper.  The authors are also very grateful to Lisa Berger, Dustin Clausen, Rafe Jones, Kiran Kedlaya, Alice Silverberg and David Zureick-Brown for many helpful discussions.   
The first author is partially supported by the Danish National Research Foundation through the Centre for Symmetry and Deformation (DNRF92).

\section{Relation to the Fermat surface} \label{to Ed section}

We will eventually reduce Theorems~\ref{main theorem 1} and \ref{main theorem 2} to statements involving only the Fermat surface $\mathcal{F}_d$.  The elliptic surface $\mathcal{E}_d$ is birational to a certain quotient of the Fermat surface $\mathcal{F}_d$ by a group $T_E$ of order $d^2$.  This section provides a more detailed description of how $\mathcal{F}_d$ and $\mathcal{E}_d$ are related.  Much of this is similar to the results of Section~5 of \cite{Ulm02}.  

We consider the Fermat surface $\mathcal{F}_d$ together with the rational map $\pi: \mathcal{F}_d \dashrightarrow \PP^1$ given by 
\[
\pi: [x_0: x_1 : x_2: x_3] \mapsto [x_0 x_1 x_2 : x_3^3].
\]  
The map $\pi$ can be resolved into a morphism by making three consecutive blow-ups at $3d$ total points (resulting in $9d$ total blow-ups).  The following lemma describes the three consecutive blow-ups at one particular point. 

\begin{lem} \label{three blow-ups}
Let $\zeta_{2d}$ denote a primitive $2d$-th root of unity.  Consider the point $x := [0: \zeta_{2d} : 1 : 0]$ in $\mathcal{F}_d$.  The rational map $\pi$ defined above can be resolved to a regular map in a neighborhood of $x$ using a series of three blow-ups.  The induced map to $\PP^1$ sends the first two exceptional divisors to $[1:0]$ and maps the third exceptional divisor bijectively to $\PP^1$.  
\end{lem}

\begin{proof}
We follow \cite[II.4.2]{Shaf94} for our explicit description of blow-ups.  Write $[t_0 : t_3]$ for coordinates on what will be the first exceptional divisor, $\PP^1$.  In a suitable neighborhood $V$ of $x$, the first-blow up is defined in $V \times \PP^1$ by $x_0 t_3 = x_3 t_0$.  The induced map $\pi$ becomes 
\[
\Big( [x_0: x_1 : x_2: x_3], [t_0: t_3] \Big) \mapsto [t_0 x_1 x_2: t_3 x_3^2].
\]
We next blow-up at the remaining point of indeterminacy, $\Big( [0: \zeta_{2d} : 1: 0], [0: 1] \Big).$  (Notice that the other points on the exceptional divisor all get mapped to $[1: 0]$, i.e., this is a fibral divisor.)  Writing $[u_0: u_3]$ for the coordinates on the next exceptional divisor $\PP^1$, so that the new blow-up has coordinates 
\[
\Big( [x_0: x_1 : x_2: x_3], [t_0: t_3], [u_0 : u_3] \Big), 
\]
constrained by the additional equation $t_0 u_3 =  x_3 u_0.$  The induced map $\pi$ becomes
\[
\Big( [x_0: x_1 : x_2: x_3], [t_0: t_3], [u_0 : u_3] \Big) \mapsto [u_0 x_1 x_2 : t_3 u_3 x_3].
\]
Similarly, after one final blow-up introducing the additional equation $u_0 v_3 = x_3 v_0$, we find
\[
\pi: \Big( [x_0: x_1 : x_2: x_3], [t_0: t_3], [u_0 : u_3], [v_0 : v_3] \Big) \mapsto [v_0 x_1 x_2 : t_3 u_3 v_3].
\]
Because we are in a neighborhood where neither $x_1$ nor $x_2$ vanish, the only possible point of indeterminacy must occur when $[v_0 : v_3] = [0 : 1]$.  From the equation $u_0 v_3 = x_3 v_0$, we deduce $[u_0 : u_3] = [0:1]$.  From the equation $t_0 u_3 =  x_3 u_0$ we deduce $[t_0 : t_3] = [0: 1]$.  But we have just shown that any point of indeterminacy of the map $\pi$ must have $t_3 u_3 v_3 = 1$, and this shows that there is no such point.  

The final exceptional divisor gets mapped as
\[
\pi: \Big( [0:\zeta_{2d} : 1: 0], [0: 1], [0 : 1], [v_0 : v_3] \Big) \mapsto [v_0 \zeta_{2d} : v_3],
\]
and this is clearly a bijective map to $\PP^1$.  
\end{proof}

\begin{rmk} \label{blow-up remark}
Of course the analogue of Lemma~\ref{three blow-ups} holds at each of the $3d$ points of indeterminacy of $\pi$.  We have phrased it in terms of a specific point $x$ for simplicity.  We write $\widehat{\mathcal{F}}_d$ for the result of these $9d$ total blow-ups.  This surface comes equipped with a \emph{morphism} $\widehat{\mathcal{F}}_d \rightarrow \PP^1$.   
\end{rmk}

\begin{defn} \label{T def}
Let $\mu_d$ denote the group of $d$-th roots of unity in $\FF_{q^2}$ and let $T = \mu_d^4/\mu_d$, where $\mu_d \leq \mu_d^4$ corresponds to the diagonal.  Let $T_E$ denote the subgroup consisting of $[t_0: t_1: t_2: t_3]$ such that $t_0 t_1 t_2 = t_3^3$; equivalently, it is the subset of elements with a representative of the form $[t_0: t_1: t_2: 1]$, where $t_0 t_1 t_2 = 1$.  These groups act on $\mathcal{F}_d$ coordinatewise, and these actions extend to $\widehat{\mathcal{F}}_d$ and the morphism $\widehat{\mathcal{F}}_d \rightarrow \PP^1$ factors through the quotient $\widehat{\mathcal{F}}_d/T_E$.  
\end{defn}

\begin{lem} \label{generic fiber lemma}
The generic fiber of $\pi: \widehat{\mathcal{F}}_d/T_E \rightarrow \PP^1$ is birational to $E_d$.  
\end{lem}

\begin{proof}
Write coordinates in $\AA^3_{\FF_{q^2}}$ as $(x,y,t)$ and consider the affine surface $X$ in $\AA^3_{\FF_{q^2}}$ determined by $y^2 + xy - t^dy = x^3$.  Define a map to $\PP^1$ by sending $(x,y,t) \mapsto [t:1]$.  It's clear that the generic fiber of this map is $E_d$, so it suffices to give a birational map from $\widehat{\mathcal{F}}_d/T_E$ to this affine surface which is compatible with the two maps to $\PP^1$.  

We first give a map from the affine open subset of $\mathcal{F}_d$ determined by $x_0x_2x_3 \neq 0$ to the affine open set in $X$ determined by $xy \neq 0$.  In terms of coordinate rings, the map $\FF_{q^2}[x,y,t,\frac{1}{xy}]/(y^2 + xy - t^dy - x^3) \rightarrow \FF_{q^2}[x_0,x_1,x_2,\frac{1}{x_0x_2}]/(x_0^d + x_1^d + x_2^d + 1)$ is defined using the same formulas as in (\ref{map to affine surface}).  This is a finite morphism of rings with module generators given by the $d^2$ elements $x_0^i x_2^j$ for $0 \leq i,j \leq d-1$ and its image lies in the subring of elements fixed by $T_E$.  In terms of fraction fields, we find a composition $K \rightarrow L^{T_E} \rightarrow L$ where $[L:K] \leq d^2$ but on the other hand $[L: L^{T_E}] \geq d^2$.  We deduce that $K \rightarrow L^{T_E}$ is an isomorphism, which completes the proof. 
\end{proof}

\begin{lem} \label{where are the blow-ups}
The elliptic surface $\mathcal{E}_d$ associated to $E_d$ can be obtained from $\pi: \widehat{\mathcal{F}}_d/T_E \rightarrow \PP^1$ by a series of blow-ups and blow-downs.  The curves to be blown down all live in fibers of $\pi$, as do the blow-ups.
\end{lem}

\begin{proof}
We have shown in Lemma~\ref{generic fiber lemma} that the generic fiber of $\widehat{\mathcal{F}}_d/T_E \rightarrow \PP^1$ is birational to $E_d$.  By the usual theory of surfaces, we may use a sequence of blow-ups to resolve $\widehat{\mathcal{F}_d}/T_E$ to a smooth surface over $\FF_{q^2}$.  Denote the result by $S$ and note that the composition $S \rightarrow \widehat{\mathcal{F}_d}/T_E \rightarrow \PP^1$ has generic fiber isomorphic to $E_d$.  

By the minimality of $\E_d$, the morphism $S \rightarrow \PP^1$ factors as a composition $\widehat{\mathcal{F}_d}/T_E \rightarrow \E_d \rightarrow \PP^1$.  As the morphism $\widehat{\mathcal{F}_d}/T_E \rightarrow \E_d$ is a birational morphism of surfaces, it is obtained as a sequence of blow-downs of divisors of self-intersection $-1$ by \cite[Theorem~II.11]{Bea96}.  The divisors that are blown-down are necessarily fibral, as otherwise the resulting map $\E_d \rightarrow \PP^1$ would not be a morphism.     
\end{proof}

\begin{lem} \label{NS for quotient}
The quotient map $\rho: \widehat{\mathcal{F}}_d \rightarrow \widehat{\mathcal{F}}_d/T_E$ induces a map on N\'eron-Severi groups, the cokernel of which is annihilated by $d^2$.  
\end{lem}

\begin{proof}
The quotient map $\rho$ is finite and surjective of degree $d^2$.  The result now follows from \cite[Exercise~7.2.2(b)]{Liu02}.
\end{proof}

We will work with the four projective surfaces $\mathcal{F}_d$, $\widehat{\mathcal{F}}_d$, $\widehat{\mathcal{F}}_d/T_E$, and $\mathcal{E}_d$, each of which comes with a rational map to $\PP^1$ which is considered part of the data.  In fact, we will work primarily with open subsets of these surfaces, so as to avoid singularities and points of indeterminacy.  Some of the relationships among these objects are illustrated in Figure~\ref{projective figure}.  The definitions are given in Definition~\ref{open defs}.

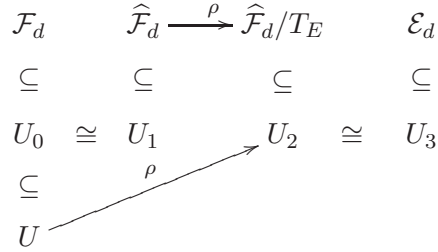
\begin{figure}[hnbt]
\[
\xymatrix{\mathcal{F}_d & \widehat{\mathcal{F}}_d \ar^{\rho}[r] & \widehat{\mathcal{F}}_d/T_E & \mathcal{E}_d \\
U_0 \ar@{}[u]|-*[@]{\subseteq} \ar@{}[r]|-*[@]{\cong} & U_1 \ar@{}[u]|-*[@]{\subseteq} & U_2 \ar@{}[u]|-*[@]{\subseteq} \ar@{}[r]|-*[@]{\cong} & U_3 \ar@{}[u]|-*[@]{\subseteq} \\
U \ar@{}[u]|-*[@]{\subseteq} \ar^{\rho}[urr]
}
\]
\caption{The relevant surfaces and open subsets.} \label{projective figure}
\end{figure}

\begin{defn} \label{open defs}
See Figure~\ref{projective figure}.
\begin{enumerate}
\item Define $U_0 \subseteq \mathcal{F}_d$ as the open set obtained by removing the $3d$ points where $x_3$ and one other coordinate are equal to zero.  
\item Define $U_1 \subseteq \widehat{\mathcal{F}}_d$ as the inverse image of $U_0$ under the map $\widehat{\mathcal{F}}_d \rightarrow \mathcal{F}_d$.  
\item As a preliminary step, define $U_2' \subseteq \widehat{\mathcal{F}}_d/T_E$ as the largest open subset obtained by subtracting certain fibral divisors in such a way that $U_2'$ is isomorphic to an open subset $U_3' \subseteq \mathcal{E}_d$.  This is possible by Lemma~\ref{where are the blow-ups}.  Then define $U_2$ to be $U_2' \cap \rho(U_1)$. 
\item Define $U_3 \subseteq \mathcal{E}_d$ to be the image of $U_2$ in $\mathcal{E}_d$ under the embedding $U_2' \hookrightarrow \mathcal{E}_d$ described above.
\item Finally, define $U \subseteq \mathcal{F}_d$ to be $U_0 \cap \rho^{-1}(U_2)$.
\end{enumerate}
\end{defn}

\begin{defn}[{\cite[Lecture~3, Section~5]{Ulm11}}] \label{L1 defn}
Given an irreducible curve $D$ on $\mathcal{E}_d$, one puts 
\[
D.E_d := D \times_{\mathcal{E}_d} E_d.
\]
This induces a map $\Div \mathcal{E}_d \rightarrow \Div E_d$.  One defines $L^1 (\Div \mathcal{E}_d)$ to be the subgroup of $\Div \mathcal{E}_d$ consisting of divisors whose images in $\Div E_d$ have degree 0.  One defines $L^1 (\NS(\mathcal{E}_d))$ as the image of $L^1 (\Div \mathcal{E}_d)$ in $\NS(\mathcal{E}_d)$.  
\end{defn}

Next we consider relations among the N\'eron-Severi groups of these spaces, as illustrated in Figure~\ref{NS figure}.  Recall that for $U \subseteq X$ an open subset, the notation $\NS(U)$ was explained in the introduction.

\begin{figure}[hnbt]
\[
\xymatrix{ 
L^1(\NS(\mathcal{F}_d)^{T_E}) & L^1(\NS(\mathcal{F}_d)) &L^1(\NS(\widehat{\mathcal{F}}_d))&L^1(\NS(\mathcal{E}_d)) \ar@{->>}[r] &E_d(\FF_{q^2}(t)) \\
L^1(\NS(U)^{T_E}) \ar@{}[r]|-*[@]{\subseteq} \ar@{}[u]|-*[@]{\subseteq} &L^1(\NS(U)) \ar@{}[u]|-*[@]{\subseteq} \ar[r] &L^1(\NS(\rho^{-1}(U_2'))) \ar@{}[u]|-*[@]{\subseteq} \ar^{\rho}[r]&L^1(\NS(U_3')) \ar@{}[u]|-*[@]{\subseteq} \ar[ur] 
}
\]
\caption{Some of the connections between our surfaces and rational points on $E_d$.} \label{NS figure}
\end{figure}

\begin{lem} \label{some surjectivities lemma} See Figure~\ref{NS figure}.  
\begin{enumerate}
\item The top-right map is the ``intersect with the generic fiber'' map as in Definition~\ref{L1 defn}.  It is surjective by \cite[Lecture~3, Theorem~5.1]{Ulm11}.
\item The diagonal map is surjective because, by Lemma~\ref{where are the blow-ups}, we have only removed fibral divisors.  
\item The cokernel of $\rho: \NS(\rho^{-1}(U_2')) \rightarrow \NS(U_3')$ is annihilated by $d^2$ as in Lemma~\ref{NS for quotient}.  Hence the same is true for $\rho: L^1(\NS(\rho^{-1}(U_2'))) \rightarrow L^1(\NS(U_3'))$.  
\item For any $t \in T_E$, two elements $v$ and $tv$ in $L^1(\NS(U))$ have the same image in $E_d(\FF_{q^2})$.  If we write $P$ for that image, then $d^2 P$ is the image of $\sum_{t \in T_E} tv$, and this latter element is in $L^1(NS(U)^{T_E})$.  
\end{enumerate}
\end{lem}

\begin{prop} \label{surjective from L1}
Figure~\ref{NS figure} describes maps from both $L^1(\NS(U))$ and $L^1(\NS(\rho^{-1}(U'_2)))$ to $E_d(\FF_{q^2})$.  
The images of these two maps are equal.  
\end{prop}

\begin{proof}
Recall that the rational map $\pi: \mathcal{F}_d \rightarrow \PP^1$ has $3d$ points of indeterminacy, and that at each such point of indeterminacy, the map $\pi$ can be resolved to a regular map by making three consecutive blow-ups; see Lemma~\ref{three blow-ups}.  
The cokernel of $\Div U \rightarrow \Div(\rho^{-1} (U'_2))$ is generated by these $3 \cdot 3d$ exceptional divisors.  For each point of indeterminacy, the first two exceptional divisors can be disregarded, because they lie in fibers of $\pi$, again by Lemma~\ref{three blow-ups}.  

To show that the third exceptional divisor can also be disregarded, we use the definition of the N\'eron-Severi group to express the class of this third exceptional divisor in terms of classes of divisors we have already considered. Let $x:= [0:  \zeta_{2d}: 1 : 0]$; the argument at other points of indeterminacy is the same.  
Consider the divisor in $\widehat{\mathcal{F}}_d \times_{\Spec \FF_{q^2}} \PP^1_{\FF_{q^2}}$ cut out by 
\[
(x_1 - \zeta_{2d} x_2)t_0 + (x_0 + x_1 + x_2)t_1,
\]  
together with the natural projection to $\PP^1$.  The fiber above $[t_0 : t_1] = [1: 0]$, which is a curve in $\widehat{\mathcal{F}}_d$, contains three exceptional divisors: the three obtained from blowing-up at the point $x$.  On the other hand, the fiber above $[0: 1]$ contains no exceptional divisors.  This shows that, as elements of $\NS(\rho^{-1}(U'_2))$, the class of the third exceptional divisor is equal to a linear combination of classes of elements corresponding to the first two exceptional divisors and classes of elements in $\NS(U)$.   This completes the proof.
\end{proof}

\begin{cor} \label{cor fixed to curve}
The map
\begin{equation} \label{equation fixed to curve}
L^1(\NS(U)^{T_E}) \otimes_{\ZZ} \ZZ\left[ d^{-1}\right] \rightarrow E_d(\FF_{q^2}(t)) \otimes_{\ZZ} \ZZ\left[ d^{-1}\right].  
\end{equation}
arising from Figure~\ref{NS figure} is surjective.
\end{cor}

\begin{proof}
This follows from combining Lemma~\ref{some surjectivities lemma} and Proposition~\ref{surjective from L1}.
\end{proof}

\section{The $\mu_d$-module structure of $\NS(\mathcal{E}_d)$} \label{representation theory section}

To prove Theorems~\ref{main theorem 1} and \ref{main theorem 2}, we produce explicit elements of $\NS(\mathcal{E}_d)$.  The results of the previous section allow us to work in $\NS(\mathcal{F}_d)$, which (in our cases) has a well-understood structure.  

\begin{defn} \label{T* def}
Recall the groups $T_E \leq T$ from Definition~\ref{T def}. These groups act on $\mathcal{F}_d$ by acting coordinate-wise.  Let $\nu$ denote a fixed isomorphism of the $d$-th roots of unity in $\FF_{q^2}$ with the $d$-th roots of unity in $\CC$.  We identify the character group $T^*$ of $T$ with tuples $(i_0, i_1, i_2, i_3)$, $i_j \in \ZZ/d\ZZ$, such that $\sum i_j = 0$; the identification of tuples to characters is given by 
\[
(i_0, i_1, i_2, i_3) \leadsto \left([t_0:t_1:t_2:t_3] \mapsto \nu(t_0)^{i_0} \nu(t_1)^{i_1} \nu(t_2)^{i_2} \nu(t_3)^{i_3}\right).
\]  
\end{defn}

\begin{prop}[{\cite[\S3]{SSV10}}] \label{decomposition over cyclotomic}
Write $V := \NS(\mathcal{F}_d) \otimes_{\ZZ} \QQ(\zeta_d)$ and for each $\lambda \in T^*$, write $V_{\lambda}$ for the subrepresentation of vectors $v \in V$ such that $tv = \lambda(t)v$ for all $t \in T$.  Each $V_{\lambda}$ is either $0$ or $1$-dimensional, and it is $1$-dimensional precisely for the trivial character and for the characters associated to tuples with all entries non-zero.  
\end{prop}

\begin{rmk}
The Fermat surface $\mathcal{F}_d$ is acted on by the projective unitary group $PU(4,q^2)$.  See \cite{HM78} for some representation-theoretic properties of this action.  In the present paper, we have found it sufficient to study the action of the abelian subgroup $T \leq PU(4,q^2)$.
\end{rmk}


The vector space $L^1(\NS(U)^{T_E}) \otimes_{\ZZ} \QQ$ appearing in Section~\ref{to Ed section} is naturally a subspace of a certain space $W$ which we now define.  

\begin{defn} \label{W def}
Define 
\[
W := \left( \NS(\mathcal{F}_d) \otimes_{\ZZ} \QQ\right)^{T_E}.
\]
\end{defn}

Notice that a tuple $(i_0, i_1, i_2, i_3) \in T^*$ as in Definition~\ref{T* def} is trivial on $T_E$ if and only if $i_0 = i_1 = i_2$.  There is a natural inclusion of $W$ into $ \left( \NS(\mathcal{F}_d) \otimes_{\ZZ} \QQ(\zeta_d)\right)^{T_E}$ and we may identify $W \otimes_{\QQ} \QQ(\zeta_d)$ with $\oplus V_{\lambda}$, where $V_{\lambda}$ is as in Proposition~\ref{decomposition over cyclotomic} and where the direct sum is taken over characters $\lambda \in T^*$ corresponding to tuples of the form $(i,i,i,*)$.  Then by Proposition~\ref{decomposition over cyclotomic} we deduce the following.
\begin{cor} 
The $\QQ(\zeta_d)$-vector space $\left( \NS(\mathcal{F}_d) \otimes_{\ZZ} \QQ(\zeta_d)\right)^{T_E}$ is $d$ or $(d-2)$-dimensional, according as whether $d \equiv 2 \bmod 3$ or $d \equiv 0 \bmod 3$, respectively.  
\end{cor}

We now compute the rank of $E_d(\FF_{q^2}(t))$.  We will use the formula \cite[Lecture~3, (5.2)]{Ulm11}
\begin{equation} \label{shioda-tate}
\Rank E_d(\FF_{q^2}(t)) = \Rank \NS(\mathcal{E}_d) - 2 - \sum (f_v - 1).
\end{equation}
Here the sum is taken over the points $v \in \PP^1$ and $f_v$ is defined to be the number of irreducible components in the fiber above $v$.  
By Lemma~\ref{where are the blow-ups}, we may work with $\NS(\widehat{\mathcal{F}_d}/T_E)$ instead of $\mathcal{E}_d$.  

\begin{prop} \label{rank of the elliptic curve}
The elliptic curve $E_d$ has rank $q$ (respectively, $q-2$) according as whether $q \equiv 1 \bmod 3$ (respectively, $q \equiv 2 \bmod 3$).  
\end{prop}

\begin{proof}
We'll prove this in the case $q \equiv 1 \bmod 3$; the other proof is identical, except that ``-2'' should be added in certain places, corresponding to $\Rank \NS(\mathcal{F}_d)^{T_E}$ being smaller.  
We use (\ref{shioda-tate}), but with $\mathcal{E}_d$ replaced by $\widehat{\mathcal{F}}_d/T_E$.  

Let $E_{x_0}$ denote the exceptional divisor coming from the third round of blow-ups described in Lemma~\ref{three blow-ups}.  (The third round is the important round because this exceptional divisor is not fibral.)  In the notation of the proof of Lemma~\ref{three blow-ups}, we see that $x_3$ is a local parameter for this blow-up.  A computation shows that $\ord_{x_3} (x_0) = 3$.  Within $\NS(\widehat{\mathcal{F}}_d/T_E)$, we have $V(x_0) = V(x_3)$.  Thus, modulo the subgroup generated by fibral divisors as well as elements from $\NS(\mathcal{F}_d)$ (note that we did not blow-up here), we have
\[
E_{x_0} + E_{x_1} + E_{x_2} = 3E_{x_0} = 3E_{x_1} = 3E_{x_2}.
\]
This shows that $\Rank(\widehat{\mathcal{F}}_d/T_E) = d + 7$, because 
$d = \Rank \left( \NS(\mathcal{F}_d)^{T_E} \right)$, and the first two rounds of blow-ups contribute 6, and the final round of blow-ups contributes 1, by the previous paragraph.  We also have $\sum (f_v - 1) = 6$.  So applying (\ref{shioda-tate}), we deduce the result.  
\end{proof}

Note that $E_{x_1}$ satisfies $E_{x_1}.E_d = O$, where $O$ is the point at infinity on $E_d$.  For example, we can take for $E_{x_1}$ the class of the scheme-theoretic closure of $O$, viewed as a point on the generic fiber of $\mathcal{E}_d$ (see \cite[Lecture~3, Section~5]{Ulm11}).  

\begin{prop} \label{compare two maps}
The map 
\[
L_{a,b} \leadsto \left(L_{a,b} - 3E_{x_1}\right).E_d
\]
is exactly the map $\phi_*$ described in (\ref{line to point}) in the introduction.  
\end{prop}

\begin{proof}
By definition, $\left(L_{a,b} - 3E_{x_1}\right).E_d$ corresponds to the unique point $R \in E_d(\FF_{q^2}(t))$, such that in $\Pic^0(E_d)$ we have 
\[
R - O \sim \left(L_{a,b} - 3E_{x_1}\right).E_d.
\]
Let $P^{(1)}_{a,b}, P^{(2)}_{a,b}, P^{(3)}_{a,b}$ denote the Galois conjugates of the point over $\FF_{q^2}(s_0)$ corresponding to $L_{a,b}$ described in the introduction.  We can view these as elements of $\Pic(E_d)$, viewed now over $\FF_{q^2}(s_0)$.  Then within $\Pic^0(E_d)_{\FF_{q^2}(s_0)}$ we have 
\[
\left(L_{a,b} - 3E_{x_1}\right).E_d = P^{(1)}_{a,b} +  P^{(2)}_{a,b} + P^{(3)}_{a,b} - 3O = \left(P^{(1)}_{a,b} - O\right) +  \left( P^{(2)}_{a,b} - O\right) + \left(P^{(3)}_{a,b}- O\right),
\]
and this latter is exactly the trace.  
\end{proof}

We now return to our analysis of $\NS(\mathcal{F}_d)$ and related vector spaces.  

\begin{defn}[Projections] \label{proj to W}
Let $V$ denote any of $\NS(\mathcal{F}_d) \otimes_{\ZZ} \QQ$ or $\NS(\mathcal{F}_d) \otimes_{\ZZ} \QQ(\zeta_d)$.  For $v \in V$, we denote by $v_W$ the element $\frac{1}{d^2}\sum_{t \in T_E} tv.$  For $\lambda \in T^*$, we denote by $v_{\lambda}$ the element $\frac{1}{|T|}\Sigma_{t\in T} \lambda^{-1}(t) tv.$  These are the projections of $v$ to $W$ and to $V_{\lambda}$, respectively.  
\end{defn}



\begin{lem} \label{composition of projections lemma}
Let $v \in \NS(\mathcal{F}_d) \otimes_{\ZZ} \QQ(\zeta_d)$ denote an element which has a non-zero projection to $V_{\lambda}$, where $\lambda \in T^*$ corresponds to a tuple $(i,i,i,d-3i)$.  Then the projection of $v$ to $W$ also has a non-zero projection to $V_{\lambda}$.  
\end{lem}

\begin{proof}
A direct computation shows that $(v_W)_{\lambda} = v_\lambda$ for $\lambda$ appearing in the decomposition of $W \otimes_{\QQ} \QQ(\zeta_d)$.  
\end{proof}

\begin{cor} \label{check non-zero projections}
Let $S \subseteq \NS(\mathcal{F}_d)$ denote a subset such that, for each $\lambda$ appearing in the decomposition of $W \otimes_{\QQ} \QQ(\zeta_d)$, some element $s \in S$ has a non-zero projection to $V_{\lambda}$.  Then $\left\{s_W \mid s \in S \right\}$ generates $W$ as a $\QQ[T/T_E]$-module. 
\end{cor}

\begin{proof}
We only need to check that the $\QQ$-span of $\left\{t \cdot s_W \mid t \in T, s \in S\right\}$ has the correct dimension, and by Lemma~\ref{composition of projections lemma}, we have the correct dimension after extending scalars to $\QQ(\zeta_d)$.
\end{proof}

We now briefly describe the strategy for proving Theorems~\ref{main theorem 1} and \ref{main theorem 2}.  As described surrounding (\ref{equation fixed to curve}), to produce generators for a full-rank subgroup of $E_d$, it suffices to find images of a $\QQ$-basis of $L^1(\NS(U)^{T_E}) \otimes_{\ZZ} \QQ$ under the composition of maps appearing in Figure~\ref{NS figure}.  We will in fact produce a basis of the bigger space $W$ from Definition~\ref{W def}, which unlike $L^1$ has no ``degree 0'' requirement.  By Corollary~\ref{check non-zero projections}, we reduce to finding a subset $S \subseteq \NS(\mathcal{F}_d)$ such that for each $\lambda \in T^*$ appearing in the decomposition of $W \otimes_{\QQ} \QQ(\zeta_d)$, some $s \in S$ has non-zero projection to $\left( W \otimes_{\QQ} \QQ(\zeta_d) \right)_{\lambda}$.  The elements $s_W$, together with their $T/T_E$-translates, will then suffice generate a full-rank subgroup of $E_d$.  On the other hand, the projection $s_W$ is a sum of translates of $s$ by $t \in T_E$, and these translates all correspond to the same element of $\Div E_d$.  

Thus the proofs of Theorems~\ref{main theorem 1} and \ref{main theorem 2} reduce to producing certain elements of $\NS(\mathcal{F}_d)$ which have non-zero projections to $(\NS(\mathcal{F}_d) \otimes_{\ZZ} \QQ(\zeta_d))_{\lambda}$ for $\lambda \in T^*$ corresponding to tuples $(i,i,i,d-3i)$ for $i, d-3i \not\equiv 0 \bmod d$, together with the trivial tuple $(0,0,0,0)$.

\section{Intersections of lines on $\mathcal{F}_d$} \label{intersection section}

The N\'eron-Severi group of a surface is equipped with a non-degenerate inner product, namely, the intersection pairing.  Our strategy is to use the simple observation that if $\langle v_{\lambda}, v_{\lambda} \rangle \neq 0$, then $v_{\lambda} \neq 0$.  The following proposition is readily verified.

\begin{prop}
Take $V$ as in Definition~\ref{proj to W}.
For any $v\in V$ one has $$\langle v_\lambda,v_\lambda \rangle =\frac{1}{|T|}\sum_{t\in T} \lambda^{-1}(t)\langle v,tv\rangle.$$  
\end{prop}

Note that in the case $v = L$ corresponds to a line, this sum takes on a very simple form because the only possible values of $\langle L,tL \rangle$ are $0, 1, 2-d$ (with $2-d$ corresponding to the self-intersection).

If one lets $\I_L$ denote the set of $t\in T$ such that $tL\neq L$ and $tL \cap L \neq \emptyset$, then one has the following particularly simple formula for $\langle L_\lambda,L_\lambda \rangle$.

\begin{prop} \label{inner product for a line}
For $L$ a line, one has
\[
\langle L_\lambda,L_\lambda\rangle=\frac{1}{|T|}\left(2-d+\sum_{t\in \I_L}\lambda^{-1}(t)\right).
\]
\end{prop}

We now study this set $\I_L$ in detail for the special case $L = L_{a,b}$, with $L_{a,b}$ as in Definition~\ref{standard line def}.

\begin{lem}
Assume $t \in T$ is such that the parametrized lines $L_{a,b}$ and $tL_{a,b}$ are distinct and intersect at the images of the same point $[u_0:v_0] \in \PP^1_{\FF_{q^2}}$.  Then $t$ has a representative in which exactly three of its four coordinates are equal to $1$.  Conversely, for any such $t$, we can find $[u_0:v_0] \in \PP^1_{\FF_{q^2}}$ such that $L_{a,b}$ and $tL_{a,b}$ intersect at the image of this point.
\end{lem}

\begin{proof}
Let's first assume that $L_{a,b}$ and $tL_{a,b}$ are distinct lines which intersect at the images of the point $[0:1] \in \PP^1_{\FF_{q^2}}$.  If we write $t = [t_0 : t_1 : t_2 : t_3]$, this implies that 
\[
[0:1:b:a] = [0:t_1 : t_2 b : t_3 a].  
\]
Thus $t \in T$ has a representative of the form $[*: 1 : 1 : 1]$, and we moreover know the first coordinate is not $1$ from the assumption that the lines are distinct.  

Now we can assume the distinct lines $L_{a,b}$ and $tL_{a,b}$ intersect at the images of a point $[1:v_0]$.  Because at most one of the elements $v_0$, $a + bv_0$ and $av_0 + b$ can be zero, a similar computation shows that $t$ has a representative in which exactly three of its entries are equal to zero.  

We now prove the second half of the lemma.  Assume $t \in T$ has a representative in which exactly three of its four coordinates are equal to $1$; say the non-identity entry occurs in the $j$-th position.  Then the lines $L_{a,b}$ and $tL_{a,b}$ are distinct and they intersect at the images of the point $[u_0: v_0]$ for which the $j$-th coordinate of $[u:v:au+bv:bu+av]$ vanishes. 
\end{proof}

Let $\pi$ denote the $q$-power Frobenius automorphism.

\begin{lem} \label{dth root lemma}
An element $x\in \FF_{q^2}$ is a $d$-th root of unity for $d=q+1$ if and only if $\pi(x)=x^{-1}$.  
\end{lem}

\begin{proof}
Note that $x^d = x \cdot \pi(x)$.
\end{proof}

\begin{lem} \label{parametrization of difficult intersections}
Assume $t \in T$ is such that the parametrized lines $L_{a,b}$ and $tL_{a,b}$ are distinct and intersect at the images of two different points $[u_0:v_0], [u_1:v_1] \in \PP^1_{\FF_{q^2}}$.  Then there exists $\gamma \in \FF_{q^2}$ such that $\tr(\gamma) \neq 0$ and $[u_0:v_0] = [\gamma:1]$ and $[u_1 : v_1] = [-\pi(\gamma):1]$, where $\tr(\gamma):= \gamma + \gamma^q$.  Conversely, any such value of $\gamma \in \FF_{q^2}$ corresponds to an intersection between $L_{a,b}$ and $tL_{a,b}$ for some unique choice of $t \in T$.  
\end{lem}

\begin{proof}
We begin by proving the first half of the lemma.  Thus we assume we are given distinct points $[u_0:v_0]$ and $[u_1:v_1]$ and $t \in T$ such that $L_{a,b}$ and $tL_{a,b}$ intersect at the images of these points in $\PP^1_{\FF_{q^2}}$.  We first note that if $v_0 = 0$, then $v_1 = 0$ and we have the same point in $\PP^1$, contrary to our assumption.  Hence we can express our points as $[u_0 : 1]$ and $[u_1 : 1]$, and we are assuming they are unequal.  We would like to show that if 
\[
\frac{u_0}{u_1}, \, \frac{au_0 + b}{au_1 + b}, \, \frac{a+bu_0}{a + bu_1}
\]
are all $(q+1)$-st roots of unity, then $u_1 = -\pi(u_0)$.  

If $\frac{u_0}{u_1}$ is a $(q+1)$-st root of unity, then $u_0^{q+1} = u_1^{q+1}$.  Using that $\pi(b) = -b$, Lemma~\ref{dth root lemma}, and the fact that $\frac{au_0 + b}{au_1 + b}$ is a $(q+1)$-st root of unity yields
\begin{align*}
(au_0^q - b)(au_0 + b) &= (au_1 + b)(au_1^q - b) \\
a^2u_0^{q+1} - abu_0 + abu_0^q - b^2 &= a^2u_1^{q+1} -abu_1 + abu_1^q - b^2 \\
\intertext{(and then, using that $u_0^{q+1} = u_1^{q+1}$)}
-abu_0 + abu_0^q &= -abu_1 + abu_1^q. 
\end{align*}
This shows that $\tr(bu_0) = \tr(bu_1)$.  Because $u_0^{q+1} = u_1^{q+1}$, we also have that $\nm(bu_0) = \nm(bu_1)$, where $\nm$ denotes the norm from $\FF_{q^2}$ to $\FF_q$.  This shows that $bu_0$ and $bu_1$ are conjugate, and we are assuming $u_0 \neq u_1$.  Hence 
\[
-b\pi(u_0) = \pi(bu_0) = bu_1.
\]
This completes the proof of the first half of the lemma.

For proving the second half of the lemma, our strategy is to consider the images of these points, $[\gamma: 1: a\gamma + b: b\gamma + a]$ and $[-\pi(\gamma): 1 : -a\pi(\gamma) + b: -b\pi(\gamma) + a]$, and to use Lemma~\ref{dth root lemma} to show that the coordinate-wise ratios are all $d$-th roots of unity.  This is a direct calculation, using that $\pi(a) = a$ and $\pi(b) = -b$.  
\end{proof}

We summarize the results of this section in the following.

\begin{prop} \label{intersection summary prop}
Fix a line $L_{a,b}$ as in Definition~\ref{standard line def}.  The following is a complete list of the intersections between $L_{a,b}$ and $tL_{a,b}$, for $t \in T$:
\begin{enumerate}
\item There is the self-intersection when $t = [1:1:1:1]$;
\item For each $t \in T$ with a representative having exactly three entries equal to $1$, the lines intersect once; \label{three entries}
\item For each $\gamma \in \FF_{q^2}$ such that $\tr(\gamma) \neq 0$, there is a unique value of $t$ such that the parametrized lines $L_{a,b}$ and $tL_{a,b}$ intersect at the images of the points $[\gamma: 1]$ and $[-\pi(\gamma):1]$ in $\PP^1_{\FF_{q^2}}$, respectively.  
\end{enumerate}
Moreover, there is no repetition among the values of $t \in T$ just enumerated.  
\end{prop}

\begin{cor} \label{trivial character}
Let $L = L_{a,b}$ denote a line as in Definition~\ref{standard line def} or the corresponding element of $\NS(\mathcal{F}_d)$.  Let $\lambda$ denote the trivial character corresponding to the tuple $(0,0,0,0)$.  Then $\langle L_{\lambda}, L_{\lambda} \rangle \neq 0$.  
\end{cor}

\begin{proof}
This follows immediately from Proposition~\ref{inner product for a line} and Proposition~\ref{intersection summary prop}.  
\end{proof}

\section{A family of multiplicative character sums}

Let $L_{a,b}$ denote a line as in Definition~\ref{standard line def} and let $v$ denote its image in $\NS(\mathcal{F}_d)$.  We will identify characters $\lambda \in T^*$ corresponding to tuples $(i,i,i,*)$ such that $\langle v_{\lambda}, v_{\lambda} \rangle \neq 0$.  In this section we work in slightly more generality and consider tuples $(i_0,i_1,i_2,i_3)$ with all entries non-zero.  We analyze the relevant inner products using multiplicative character sums.  When the tuple has the form $(i,i,i,*)$, the corresponding character sum has the form considered in the introduction; see Theorem~\ref{the character sum} below.

\begin{rmk} \label{nontrivial remark}
For the rest of this section, we assume $\lambda \in T^*$ is a \emph{non-trivial} character appearing in the decomposition of $\NS(\mathcal{F}_d) \otimes_{\ZZ} \QQ(\zeta_d)$.  In other words, we assume that $\lambda$ corresponds to a tuple $(i_0, i_1, i_2, i_3)$ with \emph{non-zero} entries.  
\end{rmk}

\begin{lem} \label{sum of -4}
Let $\I_L^0$ denote the set of $t \in T$ appearing in Proposition~\ref{intersection summary prop}(\ref{three entries}), and let $\lambda$ be a character as in Remark~\ref{nontrivial remark}.  Then
\[
\sum_{t \in \I_L^0} \lambda^{-1}(t) = -4.
\]
\end{lem}

\begin{proof}
Consider the set $S$ consisting of $t \in \I_L^0$ which are non-trivial in the $j$-th component.  Assume $\lambda^{-1}$ corresponds to the tuple $(i_0, i_1, i_2, i_3)$.  We are assuming that $i_j \neq 0$.  Then 
\[
\sum_{s \in S} \lambda^{-1}(s) = \sum_{n = 1}^{d-1} \zeta_d^{ni_j} = -1.
\] 
Because there are four choices of $j$, the result follows.
\end{proof}

The following is the main result of this section, as it relates our inner products to certain multiplicative character sums.

\begin{thm} \label{the character sum}
Take $\lambda \in T^*$ corresponding to a tuple $(i_0, i_1, i_2, i_3)$ as in Remark~\ref{nontrivial remark} and take $a,b$ as in Definition~\ref{standard line def}.  Let $v$ denote the class of $L_{a,b}$ in $\mathcal{F}_d$.  Then
\[
d^3 \langle v_{\lambda}, v_{\lambda} \rangle = -2q + \sum_{x \in \FF_{q^2}} \chi\left( x^{i_0} (x+1)^{i_1} (x + b^2)^{i_2}\right),
\]
where $\chi: \FF_{q^2}^* \rightarrow \CC^*$ is a fixed multiplicative character of order~$d$, extended to all of $\FF_{q^2}$ by setting $\chi(0) = 0$.  
\end{thm}

\begin{proof}
By Proposition~\ref{inner product for a line}, we have
\[
d^3 \langle v_\lambda, v_\lambda \rangle= 2-d+\sum_{t\in \I_L}\lambda^{-1}(t),
\]
where $\I_L$ denotes the set of elements $t \in T$ such that $\langle L_{a,b}, tL_{a,b} \rangle = 1$.  By Lemma~\ref{sum of -4} and Proposition~\ref{intersection summary prop}, this reduces to
\[
d^3 \langle v_\lambda, v_\lambda \rangle= -2-d+\sum_{\substack{\gamma \in \FF_{q^2} \\ \tr{\gamma} \neq 0}} \lambda^{-1}(t_{\gamma}),
\]
where $t_{\gamma}$ is the unique element $t \in T$ corresponding to $\gamma \in \FF_{q^2}$, $\tr(\gamma) \neq 0$, as in Lemma~\ref{parametrization of difficult intersections}.  The proof of that lemma gives an explicit description of $t_{\gamma}$ in terms of $\gamma$:
\[
t_{\gamma} = [t_0 : t_1 : t_2 : t_3] = \left[ \frac{\gamma}{-\pi(\gamma)}: 1 : \frac{a\gamma + b}{-a \pi(\gamma) + b} : \frac{a + b\gamma}{a-b\pi(\gamma)}\right].
\]
Using that $\pi(a) = a$ and $\pi(b) = -b$, we can express this more succinctly as
\[
t_{\gamma}^{-1} = \left[-(\gamma^{q-1}) : 1 : -(a\gamma + b)^{q-1} : (a+b\gamma)^{q-1}\right]. 
\]
Recall $\nu$ from Definition~\ref{T* def}.  
Our expression becomes
\[
d^3 \langle v_\lambda, v_\lambda \rangle= -2-d+\sum_{\substack{\gamma \in \FF_{q^2} \\ \tr{\gamma} \neq 0}} (-1)^{i_0 + i_2} \nu(\gamma^{q-1})^{i_0} \nu((a\gamma + b)^{q-1})^{i_2} \nu((a + b\gamma)^{q-1})^{i_3}.
\]
Next note that $\gamma \mapsto \nu(\gamma^{q-1})$ is a multiplicative character of order $d$; we denote it by $\chi$ (and extend it to all of $\FF_{q^2}$ as in the statement of the theorem).  Then the above expression becomes
\[
d^3 \langle v_\lambda, v_\lambda \rangle= -2-d+\sum_{\substack{\gamma \in \FF_{q^2} \\ \tr{\gamma} \neq 0}} (-1)^{i_0 + i_2} \chi\left(\gamma^{i_0} (a\gamma + b)^{i_2} (a + b\gamma)^{i_3}\right).
\]

This latter sum becomes more transparent if we also include values of $\gamma$ for which $\tr(\gamma) = 0$: these are exactly the values of the form $\gamma = \beta b$, where $b$ is our usual $b$ and $\beta \in \FF_q$.  Using the fact that $\chi(\alpha) = 1$ for $\alpha \in \FF_{q}^*$ (because $\chi$ has order $d$) and $\chi(b) = -1$ (because $b$ is not a $d$-th power but $b^2$ is), we find 
\[
\chi\left((\beta b)^{i_0} (a\beta b + b)^{i_2} (a + b\beta b)^{i_3}\right) = (-1)^{i_0 + i_2} \text{ or }  0,
\]
with the value of 0 occurring only for the three distinct values $\beta = 0, \frac{-1}{a}, \frac{-a}{b^2}$.  Thus our expression for the inner product can be simplified further to
\begin{align*}
d^3 \langle v_\lambda, v_\lambda \rangle &= -2-d-(q-3) + \sum_{x \in \FF_{q^2}} (-1)^{i_0 + i_2} \chi\left(x^{i_0} (ax + b)^{i_2} (a + bx)^{i_3}\right) \\
&= -2q + \sum_{x \in \FF_{q^2}} (-1)^{i_0 + i_2} \chi\left(x^{i_0} (ax + b)^{i_2} (a + bx)^{i_3}\right).\\
\intertext{Replacing $x$ by $\frac{bx}{a}$ and using again that $\chi(b) = -1$ and $\chi(\alpha) = 1$ for $\alpha \in \FF_{q}$ yields}
&= -2q + \sum_{x \in \FF_{q^2}} (-1)^{i_0 + i_2} \chi\left(\left(\frac{bx}{a}\right)^{i_0} (bx + b)^{i_2} \left(a + \frac{b^2x}{a}\right)^{i_3}\right)\\
&= -2q + \sum_{x \in \FF_{q^2}} \chi\left(x^{i_0} (x + 1)^{i_2} \left(\frac{a^2}{b^2} + x\right)^{i_3}\right)\\
&= -2q + \sum_{x \in \FF_{q^2}} \chi\left(x^{i_0} (x + 1)^{i_2} \left(1 - b^{-2} + x\right)^{i_3}\right)\\
\intertext{and then replacing $x$ by $-x-1$ yields}
&= -2q + \sum_{x \in \FF_{q^2}} \chi\left((-x-1)^{i_0} (-x)^{i_2} \left(- b^{-2} + -x\right)^{i_3}\right)\\
&= -2q + \sum_{x \in \FF_{q^2}} \chi\left((b^2x+b^2)^{i_0} (b^2x)^{i_2} \left(1 + b^2x\right)^{i_3}\right)\\
\intertext{and finally replacing $b^2 x$ by $x$ yields}
d^3 \langle v_\lambda, v_\lambda \rangle &= -2q + \sum_{x \in \FF_{q^2}} \chi\left((x+b^2)^{i_0} (x)^{i_2} \left(1 + x\right)^{i_3}\right).\\
\intertext{This expression has the desired form, although the exponents are incorrect.  To rearrange them, we perform one more substitution and replace $x$ by $\frac{-x-b^2}{x+1}$.  This transformation was chosen because it exchanges $-b^2$ and $0$, and also exchanges $-1$ and $\infty$.  Then the above becomes }
& = -2q + \sum_{x \in \FF_{q^2}} \chi\left(\left(\left(\frac{-x-b^2}{x+1}\right)+b^2\right)^{i_0} \left(\frac{-x-b^2}{x+1}\right)^{i_2} \left(1 + \left(\frac{-x-b^2}{x+1}\right)\right)^{i_3}\right) \\
\intertext{and multiplying by $1 = \chi\left((x+1)^{i_0 + i_1 + i_2 + i_3}\right)$ yields}
& = -2q + \sum_{x \in \FF_{q^2}} \chi\left(\left(-x-b^2+b^2(x+1)\right)^{i_0} (x+1)^{i_1} \left(-x-b^2\right)^{i_2} \left(x+1 + -x-b^2\right)^{i_3}\right) \\
& = -2q + \sum_{x \in \FF_{q^2}} \chi\left(\left((b^2+1)x\right)^{i_0} (x+1)^{i_1} \left(-x-b^2\right)^{i_2} \left(1 -b^2\right)^{i_3}\right) \\
d^3 \langle v_\lambda, v_\lambda \rangle &= -2q + \sum_{x \in \FF_{q^2}} \chi\left(x^{i_0} (x+1)^{i_1} \left(x+b^2\right)^{i_2}\right),
\end{align*}
as required.
\end{proof}

\begin{defn} \label{charsum defn}
Let $\chi: \FF_{q^2} \rightarrow \CC$ denote (the extension of) a multiplicative character as in Theorem~\ref{the character sum}.  We write $S_{b^2,\hat{i}}$ (or $S_{b^2}$ or $S_{\hat{i}}$ or simply $S$) for the character sum
\[
S_{b^2,\hat{i}} := \sum_{x \in \FF_{q^2}} \chi\left( x^{i_0} (x+1)^{i_1} (x + b^2)^{i_2}\right).
\]
\end{defn}

\begin{cor} \label{inner product cor}
Take notation as in Theorem~\ref{the character sum} and as in Definition~\ref{charsum defn}.  If $S \neq 2q$, then $L_{a,b}$, considered as an element of $\NS(\mathcal{F}_d)$, has non-zero projection to $\left(\NS(\mathcal{F}_d) \otimes_{\ZZ} \QQ(\zeta_d)\right)_{\lambda}$.  
\end{cor}

\begin{prop}
For any $b^2, \hat{i}$, the character sum $S_{b^2, \hat{i}}$ is real.  
\end{prop}

\begin{proof}
Replacing $\chi$ with its complex conjugate is equivalent to replacing $\chi$ by its $q$-th power, because the outputs of $\chi$ are either $0$ or (not necessarily primitive) $d$-th roots of unity.  Because $b^2 \in \FF_{q}$, this is equivalent to summing over $x^q$ for $x \in \FF_{q^2}$, and thus the sum is the same. 
\end{proof}

Recall that in Remark~\ref{nontrivial remark} we required that $i_0, i_1, i_2, i_0 + i_1 + i_2 \not\equiv 0 \bmod d$.  

\begin{prop}
We have $-2q \leq S_{b^2, \hat{i}} \leq 2q$.  
\end{prop}

\begin{proof}
This follows immediately from the Weil bound for multiplicative character sums \cite[Theorem~5.41]{LN97}.
\end{proof}

\section{Proof of Theorem~\ref{main theorem 1}}

\begin{prop} \label{mod 3 prop}
Let $\hat{i} = (i,i,i,*)$ denote a tuple with $i \neq 0$ and let $a,b$ be as in Theorem~\ref{main theorem 1}.  Let $S$ denote the associated character sum as in Definition~\ref{charsum defn}.  Then
\[
S \neq 2q.
\]  
\end{prop}

\begin{proof}
Write $b^2 = \eta$, which is a primitive $6$-th root of unity in $\FF_{q}$.  We first check that if $y:= -\eta (x+1),$ then $\chi\left( x^{i} (x+1)^{i} (x + \eta)^{i}\right) = \chi\left( y^{i} (y+1)^{i} (y + \eta)^{i}\right)$ as follows:
\begin{align*}
\chi\left( y^{i} (y+1)^{i} (y + \eta)^{i}\right) &= \chi\left( (-\eta x-\eta)^{i} (-\eta x-\eta+1)^{i} (-\eta x-\eta + \eta)^{i}\right)\\
&= \chi\left( (x+1)^{i} (x+1+\eta^2)^{i} x^{i}\right) \\
\intertext{(using that $\chi$ is identically $1$ on $\FF_{q}^*$)}
&=\chi\left( (x+1)^{i} (x+\eta)^{i} x^{i}\right) 
\end{align*}
(using that $\eta^2 - \eta + 1 = 0$).  Considering the orbits of the action $x \mapsto -\eta(x+1)$ on $\FF_{q^2}$, and noting that $u := \frac{-\eta}{1+\eta}$ is its unique fixed point, we see that
\begin{align*}
S &\equiv \chi\left( u^{i} (u+1)^{i} (u + \eta)^{i}\right) \bmod 3\ZZ[\zeta_d] \\
&\equiv 1 \bmod 3\ZZ[\zeta_d].
\end{align*}
On the other hand, $q \equiv 7 \bmod 12$, so $2q \equiv 2 \bmod 3$, so we deduce in particular that $S \neq 2q$.
\end{proof}

To prove Theorem~\ref{main theorem 1}, we need only to collect the preceding results.  

\begin{proof}[Proof of Theorem~\ref{main theorem 1}]
As described at the end of Section~\ref{representation theory section}, it suffices to show that for each character of the form $(i,i,i,*)$, the line $L_{a,b}$ has non-zero projection to $\left(\NS(\mathcal{F}_d) \otimes_{\ZZ} \QQ(\zeta_d)\right)_{\lambda}$.  For the trivial character, this follows from Corollary~\ref{trivial character}.  For the remaining characters, this follows from Corollary~\ref{inner product cor} and Proposition~\ref{mod 3 prop}.  
\end{proof}

\section{Properties of the character sum} \label{props of char sum}

In this section we gather some preliminary results concerning the character sum $S$ which will be used to prove Theorem~\ref{main theorem 2}.  

Notice that the character sum $S_{b^2,\hat{i}}$ depends only on $b^2 \in \FF_q$, not on the pair $a,b$.  When we wish to emphasize this, and to emphasize that the value does not depend on $b^2$ coming from a pair $a,b$, we write
\begin{equation} \label{S abbreviation}
S_c := S_{c,\hat{i}} := \sum_{x \in \FF_{q^2}} \chi\left( x^{i_0} (x+1)^{i_1} (x + c)^{i_2}\right).
\end{equation}

\begin{prop} \label{lidl exercise}
Fix a tuple $\hat{i} = (i_0, i_1, i_2, i_3)$. Then
\[
\sum_{c \in \FF_q} S_{c,\hat{i}} = \begin{cases} q(q-3) &\mbox{if } i_0 + i_1 \neq d\\ (q-1)^2 &\mbox{otherwise.} \end{cases}
\]
\end{prop}

\begin{proof}
(Compare Exercise~5.56 of \cite{LN97} for a similar result.)  Write $\chi_{i_0}$ for $\chi^{i_0}$, etc.  Following the notation in \cite[\S5.3]{LN97}, write 
\[
J(\chi_{i_0}, \chi_{i_1}) := \sum_{x \in \FF_{q^2}} \chi_{i_0}(x) \chi_{i_1}(1-x).
\]
By Theorems~5.21 and 5.16 of \cite{LN97}, we have
\begin{align*}
J(\chi_{i_0}, \chi_{i_1}) = \begin{cases} \frac{(-1)^{i_0} \cdot q \cdot (-1)^{i_1} \cdot q}{(-1)^{i_0+i_1} \cdot q} = q &\mbox{if } i_0 + i_1 \neq d \\ \frac{-1}{q^2} \cdot (-1)^{i_0} \cdot q \cdot (-1)^{i_1} \cdot q  = -1 &\mbox{if } i_0 + i_1 = d.\end{cases}
\end{align*}
Let $b$ be as in Definition~\ref{standard line def}.  For any $\alpha, \beta \in \FF_{q}$, write $\phi_{\alpha,\beta} := \sum_{c \in \FF_{q}} \chi_{i_2}(\alpha b + \beta + c).$  Notice that $\phi_{0,\beta} = q-1$ (because $\chi(0) = 0$ and $\chi(x) = 1$ for $x \in \FF_{q}^*$) and $\phi_{\alpha,\beta} = \phi_{\alpha', \beta}$ for any $\alpha, \alpha' \in \FF_{q}^*$ (because we may rescale by $\alpha^{-1}$, and then note that as $c$ ranges through $\FF_q$, the expression $\alpha^{-1}(\beta + c)$ ranges also through $\FF_q$).  We also have
\[
\sum_{\alpha, \beta \in \FF_q} \phi_{\alpha,\beta} = \sum_{x \in \FF_{q^2}} \sum_{c \in \FF_q} \chi(x+c) = 0.
\]
We deduce that $\phi_{\alpha,\beta} = -1$ for $\alpha \in \FF_{q}^*$.  
With these preliminary results, we are now ready to prove the proposition.
\begin{align*}
\sum_{c \in \FF_q} S_{c,\hat{i}} &= \sum_{x \in \FF_{q^2}} \chi_{i_0}(x) \chi_{i_1}(x+1) \sum_{c \in \FF_{q}} \chi_{i_2} (x+c) \\
&= \sum_{\alpha, \beta \in \FF_{q}} \chi_{i_0}(\alpha b + \beta) \chi_{i_1} (\alpha b + \beta + 1) \phi_{\alpha, \beta} \\
&= (-1) \cdot \left(\sum_{x \in \FF_{q^2}} \chi_{i_0}(x) \chi_{i_1}(x+1)\right) + (q) \left(\cdot \sum_{\beta \in \FF_q} \chi_{i_0}(0 \cdot b + \beta) \chi_{i_1}(0 \cdot b + \beta + 1)\right) \\
&= q\cdot (q-2) - \left(\sum_{x \in \FF_{q^2}} \chi_{i_0}(-x) \chi_{i_1}(1-x)\right) \\
&= q\cdot (q-2) - J(\chi_{i_0}, \chi_{i_1}).
\end{align*}
The proposition now follows from our Jacobi sum calculations at the beginning of this proof.
\end{proof}

\begin{prop} \label{c-orbit invariance}
Assume the tuple $\hat{i}$ has the form $(i,i,i,d-3i)$ and $c \in \FF_{q}^*$.  Then $S_{c} = S_{\tau(c)}$ for any of
\begin{align*}
\tau(c) &= c^{-1},\, 1-c,\, 1 - c^{-1} ,\, (1 - c)^{-1},\, (1-c^{-1})^{-1}.
\end{align*}
\end{prop}

\begin{proof}
It suffices to show the result for $\tau(c) = c^{-1}$ and $\tau(c) = 1-c$, because the other maps can be obtained as compositions of these two.  After possibly replacing $\chi$ by $\chi^i$, we can ignore the exponents; the only fact we need is that $\chi(x^d) = 1$ for any non-zero $x \in \FF_{q^2}$ (and in particular, for any $x \in \FF_q^*$).  

We first check the case $\tau(c) = c^{-1}$:
\begin{align*}
S_{c^{-1}} &= \sum_{x \in \FF_{q^2}} \chi\left( x (x+1) (x + c^{-1})\right) \\
&= \sum_{x \in \FF_{q^2}} \chi\left( cx (cx+c) (cx + 1)\right) \\
\intertext{after multiplying by $\chi(c^3) = 1$}
&= \sum_{x \in \FF_{q^2}} \chi\left( x (x+c) (x + 1)\right) \\
\intertext{after replacing $cx$ by $x$}
&= S_c.
\end{align*}

We next check the case $\tau(c) = 1-c$:
\begin{align*}
S_{1-c} &= \sum_{x \in \FF_{q^2}} \chi\left( x (x+1) (x + 1-c)\right) \\
&= \sum_{x \in \FF_{q^2}} \chi\left( (-x-1) (-x) (-x -c)\right) \\
\intertext{replacing $x$ by $-x-1$}
&= S_c.
\end{align*}
Here the last equality follows after multiplying by $\chi((-1)^3)$.
\end{proof}

\begin{defn} \label{admissible def}
We call an element $c \in \FF_{q}$ \emph{admissible} if there exist $a,b$ as in Definition~\ref{standard line def} such that $c = b^2$.  
\end{defn}

\begin{lem} \label{count admissible values}
Assume $q \equiv 1 \bmod 4$.  Then exactly $\frac{q-1}{4}$ of the values of $c \in \FF_{q}$ are admissible.
\end{lem}

\begin{proof}
This is well-known; see for example Chapter~5, Exercises 29-31 of \cite{IR90} for a similar result in the case $q = p$.  
\end{proof}

\begin{lem} \label{c-orbits}
Assume $q \equiv 1 \bmod 4$ and assume $c$ is an admissible value as in Definition~\ref{admissible def}.  Then $(1-c^{-1})^{-1}$ is also admissible while $c^{-1}, (1-c^{-1}), 1-c, (1-c)^{-1}$ are not admissible.
\end{lem}

\begin{proof}
By definition of \emph{admissible}, we know that $c$ is a quadratic nonresidue (NR) and $c-1$ is a quadratic residue (QR).  Because $-1$ is a QR and $c^{-1}$ is an NR, we deduce that $c-1$ is a QR implies $1 - c^{-1}$ is an NR implies $c^{-1} - 1$ is an NR.  Hence in particular $c^{-1}$ is not admissible. Because $-c^{-1} = (1-c^{-1}) - 1$ is an NR, we deduce that $1-c^{-1}$ is not admissible.  Also $c-1$ being a QR implies that both $1-c$ and $(1-c)^{-1}$ are QRs, and in particular, neither of these is admissible.  

On the other hand, we've already checked that $1-c^{-1}$ is an NR, hence so is $(1-c^{-1})^{-1}$.  Also, $c^{-1} = 1-(1-c^{-1})$ is an NR, hence the product $\left(1-(1-c^{-1})\right) (1-c^{-1})^{-1} = (1-c^{-1})^{-1} - 1$ is a QR.  Together this shows that $(1-c^{-1})^{-1}$ is admissible, which completes the proof. 
\end{proof}

We will prove Theorem~\ref{main theorem 2} by considering the sum $\sum_{c \in \FF_q} S_c$, where $S_c$ is as in (\ref{S abbreviation}) (and some fixed tuple $\hat{i}$ is implicit).  The values $S_0$ and $S_1$ are special, because in those cases our character sum reduces to a Jacobi sum.

\begin{lem} \label{jacobi lemma}
We have $S_0 = \begin{cases}  q & \mbox{if }i_0 + i_2 \neq d \\ -1 & \mbox{if } i_0 + i_2 = d\end{cases}$ and $S_1 = \begin{cases}  q & \mbox{if }i_1 + i_2 \neq d \\ -1 & \mbox{if } i_1 + i_2 = d.\end{cases}$
\end{lem}

\begin{proof}
We prove the result for $S_0$ only.  We have
\[
S_0 = \sum_{x \in \FF_{q^2}} \chi_{i_0 + i_2}(x) \chi_{i_1}(x+1) = \sum_{x \in \FF_{q^2}} \chi_{i_0 + i_2}(x) \chi_{i_1}(1-x).  
\]
If $i_0 + i_2 \neq d$, we are finished as in the proof of Lemma~\ref{lidl exercise}.  If $i_0 + i_2 = d$, then we have 
\[
S_0 = \sum_{x \in \FF_{q^2}} \chi_{d}(x) \chi_{i_1}(1-x) = \sum_{x \in \FF_{q^2},\, x \neq 0} \chi_{i_1}(1-x) = -1 + \sum_{x \in \FF_{q^2}} \chi_{i_1}(1-x) = -1.
\]
\end{proof}

\section{Proof of Theorem~\ref{main theorem 2}}

Exactly as in the proof of Theorem~\ref{main theorem 1}, it suffices to show that for the trivial character $\lambda = (0,0,0,0)$ and for each character $(i,i,i,d-3i)$ for $i, 3i \neq 0 \bmod d$, there is some line $L_{a,b}$ with non-zero projection to $V_{\lambda}$.  The trivial character is accounted for in Corollary~\ref{trivial character}. For the rest of the tuples $\hat{i}$, it suffices to show that some admissible character sum satisfies $S_{c, \hat{i}} \neq 2q$.  Notice that if $\gcd(k,d) = 1$, and if $\sigma$ is the automorphism of $\ZZ[\zeta_d]$ induced by $\zeta_d \mapsto \zeta_d^k$, then we have $S_{c,(ki,ki,ki,d-3ki)} = \sigma(S_{c,(i,i,i,d-3i)})$.  Because $2q$ is fixed by any automorphism, we deduce that at most $n$ lines are needed, where $n$ is equal to the number of divisors of $d$.  In fact, $n-1$ lines suffice, because $i = d$ corresponds to the trivial character.  In summary, to prove Theorem~\ref{main theorem 2}, it suffices to prove the following proposition.  

\begin{prop} \label{exists line prop}
Let $i$ be such that $i,-3i \not\equiv 0 \bmod d$.  Then there exists some line $L_{a,b}$ as in Definition~\ref{standard line def} such that the character sum associated to $\hat{i} = (i,i,i,d-3i)$ satisfies $S \neq 2q$.
\end{prop}

\begin{proof}[Proof of Proposition~\ref{exists line prop}]
In the terminology of Definition~\ref{admissible def}, it suffices to show that there is some admissible $c \in \FF_q$ such that $S_{c, \hat{i}} \neq 2q$.  We prove this by contradiction.  Thus assume all $\frac{q-1}{4}$ admissible $c$-values give a value of $2q$.  Let $c$ denote an admissible value and assume first that $c = (1-c^{-1})^{-1}$.  (We know automatically that $c \neq c^{-1}, 1-c, 1-c^{-1}, (1-c)^{-1}$ because none of those are admissible.)  This implies $c = 2$.  This means that $1-c = (1-c)^{-1}$ and $c^{-1} = 1-c^{-1}$.  We claim that this is the only circumstance in which $c$ is admissible and the six values $c, \,c^{-1},\, 1-c,\, 1 - c^{-1} ,\, (1 - c)^{-1},\, (1-c^{-1})^{-1}$ are not all distinct.  As in the proof of Lemma~\ref{c-orbits}, we have that $c^{-1}$ and $1-c^{-1}$ are quadratic non-residues and $1-c$ and $(1-c)^{-1}$ are quadratic residues.  If we have equality for either of these two pairs, then $c = 2$.

Let $c$ denote an admissible value and consider the six values $c, \,c^{-1},\, 1-c,\, 1 - c^{-1} ,\, (1 - c)^{-1},\, (1-c^{-1})^{-1}$.  By Proposition~\ref{c-orbit invariance}, these six (not necessarily distinct) elements of $\FF_q$ all produce the same character sum.  If $c = (1-c^{-1})^{-1}$, then we have exactly one admissible value and two inadmissible values in this orbit.  Otherwise we have exactly two admissible values and four inadmissible values in this orbit.  If each of the $\frac{(q-1)}{4}$ admissible $c$-values (Lemma~\ref{count admissible values}) produces a character sum equal to the upper bound, then the preceding remarks in fact yield $\frac{3(q-1)}{4}$ values of $c$ (one-third admissible, two-thirds inadmissible) such that $S_c = 2q$.  Combining these assumptions with Lemma~\ref{jacobi lemma} and the Weil bound, we find
\[
\sum_{c \in \FF_q} S_c  \geq \frac{3(q-1)}{4} \cdot 2q + (-1) + (-1) + \left( \frac{q-1}{4} - 2\right) \cdot (-2q) = q(q-1)+4q-2 > \max(q(q-3), (q-1)^2).
\]
The latter contradicts Proposition~\ref{lidl exercise}, and so we deduce that at least some admissible $c$-value misses the upper bound.  
\end{proof}




\section{Further directions}

There are many directions in which the results of this paper could be extended.  We conclude by enumerating several of these.

\begin{enumerate}
\item In many cases, Theorems~\ref{main theorem 1} and \ref{main theorem 2} describe explicit rational generators for $E_d(\FF_{q^2}(t))$.  What can be said about the index of the subgroup of $E_d(\FF_{q^2}(t))$ that our explicit points generate?  Because of the need to invert $d$ in Corollary~\ref{cor fixed to curve}, our techniques are well-suited only to detecting the part of the index relatively prime to $d$.  
\item The explicit points produced by our method are typically not concise; see Example~\ref{example of explicit point}. Understanding the relations among the points given by various lines will hopefully lead to a more explicit rational generating set of lower height and smaller index as found with respect to the Legendre curve in (3.1) of \cite{Ulm14}.  Further possible applications include computing Tate-Shafarevich groups, as in \cite{Ulm14b}.  
\item We have considered only one family of lines on the Fermat surface.  What do we gain by considering other lines (such as those described in \cite[\S 5.3]{SSV10}) or higher degree curves?  
\item We use only the small abelian piece $T$ of the automorphism group of the Fermat surface.  What do we gain by considering all automorphisms?  
\end{enumerate}

For our applications to elliptic curves, we found lines which generated a portion of the N\'eron-Severi group of the Fermat surface.  This was the portion corresponding to tuples of the form $(i,i,i,*)$; see Section~\ref{representation theory section} for our first description of these tuples.   There are many interesting directions for future research if we consider all tuples appearing in the decomposition of the N\'eron-Severi group.  Such questions are very interesting with regards to the Fermat surface and the character sums themselves; they are presumably unnecessary for applications to elliptic curves.    

\begin{enumerate}[resume]
\item Computations suggest that our line from Theorem~\ref{main theorem 1}, together with its $(\mu_d^4/\mu_d)$-translates, often rationally generates the entire N\'eron-Severi group of the Fermat surface.  For example, this is the case for every prime $p \equiv 7 \bmod 12$ from $7$ through $127$.  It is not true for $p = 139$.  We would like to determine when this happens.  Based on computational evidence, we expect that the line and its translates usually generate the entire N\'eron-Severi group.
\item The question of what part of the N\'eron-Severi group of the Fermat surface is generated by a particular line can be phrased entirely as a question of how often multiplicative character sums as in Definition~\ref{charsum defn} hit their upper bound of $2q$; this latter question seems interesting in its own right.  
\end{enumerate}

\bibliography{doug}
\bibliographystyle{plain}

\end{document}